\documentclass[11pt]{amsart}
\usepackage{multirow,bigdelim}
\usepackage{amsfonts}
\usepackage{dsfont}
\usepackage{amsmath}
\usepackage{mathrsfs}
\usepackage{todonotes}
\usepackage{hyperref}
\usepackage{multimedia}
\usepackage{graphicx}
\usepackage{indentfirst}
\usepackage{bm}
\usepackage{amsthm}
\usepackage{mathrsfs}
\usepackage{latexsym}
\usepackage{amsmath}
\usepackage{amssymb}
\usepackage{microtype}
\usepackage{relsize}
\usepackage{hyperref}

\DeclareSymbolFont{SY}{U}{psy}{m}{n}
\DeclareMathSymbol{\emptyset}{\mathord}{SY}{'306}

\theoremstyle{plain}

\newtheorem{thm}{Theorem}[section]
\newtheorem{cor}[thm]{Corollary}
\newtheorem{lem}[thm]{Lemma}

\newtheorem{defn}[thm]{Definition}

\theoremstyle{definition}

\newtheorem{ex}[thm]{Example}

\numberwithin{equation}{section}

\begin{document}
\title[On the $N$-hypercontractions and similarity of multivariable weighted shifts]{On the $N$-hypercontractions and similarity of multivariable weighted shifts}

\author{Yingli Hou} \author{Shanshan Ji} \author{Jing Xu}
\email{houyingli0912@sina.com, jishanshan15@outlook.com, xujingmath@outlook.com}
\address{School of Mathematical Sciences, Hebei Normal University, Shijiazhuang, Hebei 050016, China}

\thanks{This work was supported by the National Natural Science Foundation of China, Grant No. 11922108 and 12001159. }

\subjclass[2000]{Primary 47C15, 47B13; Secondary 47B48, 47L40}

\keywords{The class $\mathcal{B}_{n}^{m}(\Omega)$; Similarity; $N$-hypercontraction; Curvature inequality}

\begin{abstract}
 In \cite{SH}, A. L. Shields proved a well-known theorem for the similarity of unilateral weighted shift operators. By using the generalization of this theorem for multivariable weighted shifts and  the curvature of holomorphic bundles,  we  give a necessary and sufficient condition for the similarity of  $m$-tuples in Cowen-Douglas class. We also present a necessary condition for commuting $m$-tuples of backward weighted shift operators to be $n$-hypercontractive in terms of the weight sequences.
\end{abstract}

\maketitle

\section{Introduction}

For $m\geq1$, let $\Omega$ be a bounded connected open subset of $m$-dimensional complex space $\mathbb{C}^{m}$, $\mathcal{L}(\mathcal{H})^{m}$ be the space of all commuting $m$-tuples $\mathbf{T}=(T_{1},\cdots,T_{m})$ of bounded linear operators on the complex separable Hilbert space $\mathcal{H}$ and $\mathcal{L}(\mathcal{H})^{1}$ is written as $\mathcal{L}(\mathcal{H})$.
Let $\mathbf{T}=(T_{1},\cdots,T_{m}),\mathbf{ S}=(S_{1},\cdots,S_{m})\in\mathcal{L}(\mathcal{H})^{m}.$ If there is a unitary $U$ such that $UT_{i}=S_{i}U,1\leq i\leq m$, then $\mathbf{T}$ and $\mathbf{S}$ are unitarily equivalent (denoted by $\mathbf{T}\sim_{u}\mathbf{S}$).
If there is an invertible operator $X$ such that $XT_{i}=S_{i}X,1\leq i\leq m$, then $\mathbf{T}$ and $\mathbf{S}$ are similar (denoted by $\mathbf{T}\sim_{s}\mathbf{S}$). 
However, it is not easy to characterize the similarity (unitary equivalence) of any two commuting tuples of operators. Thus one can only consider the similarity classification in some special classes.

In \cite{CD,CD2}, Cowen and Douglas introduced the class $\mathcal{B}_{n}^{m}(\Omega)$ and proved that two $m$-tuples $\mathbf{T}=(T_{1},\cdots,T_{m})$ and $\mathbf{ S}=(S_{1},\cdots,S_{m})$ in $\mathbf{\mathcal{B}}_{n}^{m}(\Omega)$ are unitarily equivalent if and only if the vector bundles $E_\mathbf{T}$ and $E_{\mathbf{S}}$ are equivalent as Hermitian holomorphic vector bundles.
They also provide a large number of local criteria for holomorphic bundle equivalence. For the Hermitian bundle associated with the Cowen-Douglas class, the unitary transformation preserves its local rigidity, so the local properties of the bundles are also valid as a whole.
In particular, the curvature and its covariant derivative of Hermitian holomorphic bundles corresponding to tuples in $\mathcal{B}_{n}^{m}(\Omega)$ are shown as a set of complete unitary invariants.
Thus the curvature contains a lot of information about tuples and is closely related to the unitary classification of tuples.

An effective way to study operators is to establish a model for operators with some common properties. There are rich theories of normal operators and subnormal operators, such as the spectrum theorem, von Neumann-Wold theorem and so on.
In order to generalize a famous model theorem on contractions given by Sz.-Nagy and Foias, Alger in \cite{A1} proposed the concept of $n$-hypercontraction, which is stronger than contraction.
For the Cowen-Douglas contraction $T$ with index one, the curvature of $T$ is dominated by the adjoint curvature of multiplication operator on Hardy space in \cite{M}.
Then this result is extended to tuples in \cite{MS1,MS2}.
So the $n$-hypercontraction of the tuple will affect its curvature.
Based on Sz.-Nagy-Foias theory in 1993, M\"{u}ller and Vasilescu gave the following theorem in \cite{V} to decide when a commuting $m$-tuple of operators is unitarily equivalent to the restriction of backward weighted shift on some vector-valued space.

\begin{lem}\cite{V}\label{1}
Let $\mathbf{T}=(T_{1},\cdots,T_{m})\in\mathcal{L}(\mathcal{H})^{m}$ be a commuting $m$-tuple of operators and $n\geq1$ be an integer. Then there exist a Hilbert space $E$ and an $\mathbf{S}^{*}_{n,E}$-invariant subspace $K$ of $H_{n,E}^{2}$ such that $\mathbf{T}$ is unitarily equivalent to $\mathbf{S}^{*}_{n,E}|K$
if and only if $\mathbf{T}$ is an $n$-hypercontraction with $\lim\limits_{k\rightarrow\infty}\mathbf{M}_{\mathbf{T}}^{k}(1)=0$ in the strong operator topology.
\end{lem}

In \cite{CS2}, Curto and Salinas pointed out that every $m$-tuple $\mathbf{T}\in \mathcal{B}_{n}^{m}(\Omega)$ can be realized as the adjoint of an $m$-tuple of multiplication operators by coordinate functions on some Hilbert space of holomorphic functions.
In the past four decades, Cowen, Douglas, Misra and many other mathematicians have done a lot of work on the classification of this class of operators (see \cite{CM2,DKT,HJK,CK,JJDG,JKX,JJD,KT,M,ZKH}).

The class $\mathcal{B}_{1}^{1}(\mathbb{D})$ contains many the adjoint of unilateral weighted shift operators.
Shields provided a large number of properties and results about these shifts in \cite{SH}, including the equivalent condition for the similarity of such operators.

\begin{lem}\cite{SH}\label{2}
Let $T_1$ and $T_2$ be unilateral weighted shifts, 
and their weight sequences are $\{\lambda_j\}_{j=0}^{\infty}$ and $\{\tilde{\lambda}_j\}_{j=0}^{\infty}$, respectively. Then $T_1\sim_{s}T_2$ if and only if there exist positive constants $C_1$ and $C_2$ such that
$0<C_1\leq \big |\frac{\lambda_k\lambda_{k+1}\cdots \lambda_l}{\tilde{\lambda}_k\tilde{\lambda}_{k+1}\cdots \tilde{\lambda}_l} \big | \leq C_2$
for all $0\leq k\leq l$.
\end{lem}

In order to find new invariants of shift operators, Clark and Misra first studied the similarity of backward weighted shift operators by using the quotient of the metric of associated holomorphic bundles in \cite{CM2,CM1,CM3}.
In a way, this result can be seen as a geometric version of the similar result due to Shields (Lemma \ref{2}).
Subsequently, the study of geometric similarity invariants of the weighted shift operator has become an important research branch of this subject (see \cite{DKT,HJK,JKX,KT}).
The notion of a weighted shift has a natural generalization to commuting tuples of operators.
However, there are few related studies on the similarity of tuples.

Inspired by the above results,  we extend Lemma \ref{2} to the case of commuting $m$-tuples of operators (see Theorem \ref{c4.2}). By using this theorem and Lemma \ref{1}, we characterize the similarity of $n$-hypercontractive commuting $m$-tuples in $\mathbf{\mathcal{B}}_{1}^{m}(\Omega)$ in terms of the difference of the curvatures (see Theorem \ref{3}).
Thus a natural question is the following: when a tuple is $n$-hypercontractive.
In the last section, we use the weight sequence to give a necessary condition of some $m$-tuple to be $n$-hypercontractive (see Theorem \ref{thm2}), and show that the $n$-hypercontraction assumption in Theorem \ref{3} is also necessary (see Example \ref{maincor}). The following are two main theorems of this paper.

\begin{thm}\label{c4.2}
Let $n$ be a positive integer and $\mathbf{T}=(T_{1},\cdots,T_{m}),\mathbf{S}=(S_{1},\cdots,S_{m})\in \mathcal{B}_{1}^{m}(\mathbb{B}^{m})$ be tuples of backward weighted shifts on Hilbert spaces with reproducing kernels
$\widetilde{K}(z,w)=(1-\langle z,w\rangle)^{-n}$ and $K(z,w)=\sum\limits_{i=0}^{\infty}a(i)(z_{1}\overline{w}_{1}+\cdots+z_{m}\overline{w}_{m})^{i}, a(i)>0$, respectively. If $\mathbf{S}$ is $n$-hypercontractive, then $\mathbf{T}$ is similar to $\mathbf{S}$ if and only if there exists a bounded plurisubharmonic function $\psi$ such that
  $$\mathcal{K}_\mathbf{T}(w)-\mathcal{K}_\mathbf{S}(w)=\sum \limits_{i,j=1}^{m}\frac{\partial^{2}\psi(w)}{\partial w_{i}\partial \overline{w}_{j}}dw_{i}\wedge d\overline{w}_{j},\quad w\in\mathbb{B}^{m}.$$
\end{thm}

\begin{thm}
Let $m\geq2$ be a positive integer and $\mathbf{T}=(T_{1},\cdots,T_{m})$ be a commuting $m$-tuple of backward weighted shifts on Hilbert space $\mathcal{H}$ with reproducing kernel
$K(z,w)=\sum\limits_{\alpha\in\mathbf{Z}_{+}^{m}}\rho(\alpha)z^{\alpha}\overline{w}^{\alpha}$. If $\mathbf{T}$ is $n$-hypercontractive, then for any non-zero
$\alpha=(\alpha_{1},\cdots,\alpha_{m})\in\mathbf{Z}_{+}^{m}$, $$\sum\limits_{\mbox{\tiny$\begin{array}{c}
 \beta\in{\mathbf{Z}}_{+}^{m}\\
\beta\leq\alpha \\
|\alpha-\beta|=1\end{array}$}}\frac{\rho(\beta)}{\rho(\alpha)}\leq\frac{|\alpha|}{|\alpha|+n-1}.$$
\end{thm}

\section{\sf Preliminaries}

Let $\mathbf{T}=(T_{1},\cdots,T_{m})\in\mathcal{L}(\mathcal{H})^{m}$ denote the $m$-tuple of operators such that $T_{i}T_{j}=T_{j}T_{i},1\leq i,j\leq m$ will be designated as a commuting $m$-tuple. Let $\mathbf{Z}_{+}^{m}$ be the collection of $m$-tuples of nonnegative integers.
For $\alpha=(\alpha_{1},\cdots,\alpha_{m})\in \mathbf{Z}_{+}^{m}$, setting $|\alpha|=|\alpha_{1}|+\cdots+|\alpha_{m}|$, $\alpha!=\alpha_{1}!\cdots\alpha_{m}!$, $\mathbf{T}^{\alpha}=T_{1}^{\alpha_{1}}\cdots T_{m}^{\alpha_{m}}$ and $\mathbf{T}^{*}=(T_{1}^{*},\cdots,T_{m}^{*}).$ For any $\alpha, \beta\in \mathbf{Z}_{+}^{m}$, define $\alpha+\beta=(\alpha_{1}+\beta_{1},\cdots,\alpha_{m}+\beta_{m})$ and $\alpha\leq\beta$ whenever $\alpha_{i}\leq\beta_{i}, 1\leq i\leq m$.

\subsection{The Cowen-Douglas class $\mathbf{\mathcal{B}}_{n}^{m}(\Omega)$}

\begin{defn} \cite{CD,CD2}
For $\Omega$ a connected open subset of $\mathbb{C}^{m}$ and $n$ a positive integer, let $\mathbf{\mathcal{B}}_{n}^{m}(\Omega)$ denote the Cowen-Douglas class of commuting $m$-tuples  $\mathbf{T}=(T_{1},T_{2},\cdots,T_{m})\in\mathcal{L}(\mathcal{H})^{m}$ satisfying: \begin{itemize} \item [(1)]$ran(\mathbf{T}-w)$ is closed for all $w$ in $\Omega$; \item [(2)]$\bigvee \limits_{w{\in}{\Omega}} \ker(\mathbf{T}-w)=\mathcal H$; and \item [(3)]$\dim \ker (\mathbf{T}-w)=n$ for $w$ in $\Omega$, \end{itemize}
where $\ker(\mathbf{T}-w)$ means the joint kernel $\bigcap_{i=1}^m\ker(T_i-w_i)$ for $w=(w_{1},w_{2},\cdots,w_{m})$.
\end{defn}

For a tuple $\mathbf{T}$ in $\mathbf{\mathcal{B}}_{n}^{m}(\Omega),$ let $(E_{\mathbf{T}},\pi)$ denote the sub-bundle of the trivial bundle $\Omega\times \mathcal{H}$ defined by $$E_\mathbf{T}=\{(w, x)\in \Omega\times {\mathcal H}: x \in \ker (\mathbf{T}-w)\},\quad \pi(w,x)=w.$$ Two commuting tuples $\mathbf{T}$ and $\widetilde{\mathbf{T}}$ in $\mathbf{\mathcal{B}}_{n}^{m}(\Omega)$ are unitarily equivalent if and only if the vector bundles $E_\mathbf{T}$ and $E_{\widetilde{\mathbf{T}}}$ are equivalent as Hermitian holomorphic vector bundles.
Since $\dim \ker (\mathbf{T}-w)=n$ for all $w$ in $\Omega$, the rank of the holomorphic vector bundle $E_\mathbf{T}$ is $n$. Let $\sigma=\{\sigma_{1},\sigma_{2},\ldots,\sigma_{n}\}$ be the holomorphic frame of $E_\mathbf{T}$ and form the metric of inner products $h(w):=(\langle \sigma_j(w),\sigma_i(w \rangle)_{i,j=1}^{n},\, w=(w_{1},\cdots,w_{m})\in\Omega,$
then the curvature $\mathcal{K}_\mathbf{T}$ of the bundle $E_{\mathbf{T}}$ is given by the following formula \begin{align}\label{curvature} \mathcal{K}_\mathbf{T}(w)&:=\sum \limits_{i,j=1}^{m}\frac{\partial}{\partial \overline{w}_{j}}\big(h^{-1}(w)\frac{\partial}{\partial w_{i}} h(w)\big)d\overline{w}_{j}\wedge dw_{i}. \end{align}
In fact, $\mathcal{K}_\mathbf{T}(w)$ depends on the choice of holomorphic frame $\sigma$.
We also use $\mathcal{K}_\mathbf{T}(\sigma)(w)$ to represent the curvature function of $E_{\mathbf{T}}$ associated with the frame $\sigma$.
In particular, when $\mathbf{T}\in \mathcal{B}_{1}^{m}(\Omega)$, the curvature of the bundle $E_{\mathbf{T}}$ can be defined as $\mathcal{K}_\mathbf{T}(w)=-\sum \limits_{i,j=1}^{m}\frac{\partial^{2}\log\|\gamma(w)\|^{2}}{\partial w_{i}\partial \overline{w}_{j}}dw_{i}\wedge d\overline{w}_{j},$ where $\gamma$ is a non-vanishing holomorphic section of $E_\mathbf{T}$.

Let $\mathbf{T}=(T_1,\cdots,T_m)\in \mathcal{B}_n^m(\Omega)$. Then $m$-tuple $\mathbf{T}$ is unitarily equivalent to
the adjoint $\mathbf{M}_z^*=(M_{z_1}^*,M_{z_2}^*,\cdots,M_{z_m}^*)$ of an $m$-tuple of multiplication operators by coordinate functions on a Hilbert
space $\mathcal{H}$ of holomorphic functions on $\Omega^*=\{w\in \mathbb{C}^m:\bar{w}\in \Omega\}$ with reproducing kernel $K$ due to Curto and Salinas in \cite{CS2} and
$\ker(\mathbf{M}_z^*-w)=\{K(\cdot, \bar{w})\xi, \xi\in \mathbb{C}^n\}$. If $K(z,w)=\sum\limits_{\mbox{\tiny$\begin{array}{c}
 \alpha\in{\mathbf{Z}}_{+}^{m}\end{array}$}}\rho(\alpha)z^\alpha w^\alpha$ for some $\rho(\alpha)>0$, then an orthonormal basis of $\mathcal{H}$ can be determined as
$\{\mathbf{e}_{\alpha}(z):\mathbf{e}_{\alpha}(z)=\sqrt{\rho(\alpha)}z^{\alpha}\}_{\alpha\in\mathbf{Z}_{+}^{m}}$. For any
$z=(z_1,z_2,\cdots,z_m)\in\Omega$ and $\alpha=(\alpha_1,\alpha_2,\cdots,\alpha_m)\in\mathbf{Z}_{+}^{m}$, we have
$$M_{z_i}\mathbf{e}_{\alpha}(z)
=\sqrt{\rho(\alpha)}z^{\alpha+e_i}
=\sqrt{\frac{\rho(\alpha)}{\rho(\alpha+e_i)}}\sqrt{\rho(\alpha+e_i)}z^{\alpha+e_i}
=\sqrt{\frac{\rho(\alpha)}{\rho(\alpha+e_i)}}\mathbf{e}_{\alpha+e_i}(z).$$
It follows that $\mathbf{T}$ can be realized as a commuting $m$-tuple of backward weighted shift up to unitary equivalence.

\subsection{$N$-hypercontractivity of commuting tuples}

In the case of a single operator, it can be seen from the model theorem \cite{N} and similarity theorems \cite{KT,DKT} that the $n$-hypercontraction of the operator is closely related to the similarity of the operator. Therefore, when considering the similarity of commuting tuples, it is necessary to introduce the following concept.

We will use the symbols in \cite{V} to define the $n$-hypercontractive of commuting tuples.
Let $\mathbf{T}=(T_{1},\cdots,T_{m})\in\mathcal{L}(\mathcal{H})^{m}$. Define the operator
$\mathbf{M}_{\mathbf{T}}:\mathcal{L}(\mathcal{H})\rightarrow\mathcal{L}(\mathcal{H})$
by $$\mathbf{M}_{\mathbf{T}}(X):=\sum\limits_{i=0}^m T_{i}^*XT_{i}$$
for any operator $X$ in $\mathcal{L}(\mathcal{H})$.

\begin{defn}\label{def}
A tuple $\mathbf{T}=(T_{1},\cdots,T_{m})\in\mathcal{L}(\mathcal{H})^{m}$ is called $n$-hypercontraction, if
$\triangle_{\mathbf{T}}^{(k)}:=(I-\mathbf{M}_{\mathbf{T}})^{k}(I)\geq 0$
for the identity $I$ of $\mathcal{H}$ and $1 \leq k \leq n$.  The special case of $1$-hypercontraction corresponds to the usual (row) contraction.
\end{defn}

Note that
$\mathbf{M}_{\mathbf{T}}^{k}(X)=\sum\limits_{\alpha\in \mathbf{Z}_{+}^{m}\atop|\alpha|=k}\frac{k!}{\alpha!} \mathbf{T}^{*\alpha}X\mathbf{T}^{\alpha}$ for any nonnegative integer $k,$
then we have
$$\triangle_{\mathbf{T}}^{(k)}=\sum\limits_{j=0}^k (-1)^j{k \choose j}\mathbf{M}_{\mathbf{T}}^{j}(I)=\sum \limits_{|\alpha|\leq k}(-1)^{|\alpha|}\frac{k!}{\alpha!(k-|\alpha|)!}\mathbf{T}^{*\alpha}\mathbf{T}^{\alpha}.$$
Let $\mathbb{B}^{m}$ be the open unit ball $\{w : |w|<1\}$ in $\mathbb{C}^{m}$. The space $H_{n}^{2}$ is a reproducing kernel Hilbert space with kernel function
$$K(z,w)=\frac{1}{(1-\langle z, w\rangle)^{n}}=\sum \limits_{\alpha\in \mathbf{Z}_{+}^{m}}\frac{(n+|\alpha|-1)!}{\alpha!(n-1)!}z^{\alpha}\bar{w}^{\alpha},\quad z,w\in(\mathbb{B}^{m})^{*},$$
where $\langle z, w\rangle=z_{1}\overline{w}_{1}+\cdots+z_{m}\overline{w}_{m}.$ Letting
$\rho_{n}(\alpha)=\frac{(n+|\alpha|-1)!}{\alpha!(n-1)!},$
we obtain an orthonormal basis of space $H_{n}^{2}$ as $\{\mathbf{e}_{\alpha}(z):\mathbf{e}_{\alpha}(z)=\sqrt{\rho_{n}(\alpha)}z^{\alpha}\}_{\alpha\in\mathbf{Z}_{+}^{m}}$.
In particular, when $n = 1$, it is the Drury-Arveson space $H^{2}$.
Obviously, the adjoint of the multiplication operators by coordinate functions on $H_{n}^{2}$ is $n$-hypercontractive.

\subsection{Plurisubharmonic function}
The following are some basics of plurisubharmonic functions, which will be used to describe the similarity of some $m$-tuples in $\mathcal{B}_{1}^{m}(\mathbb{B}^{m})$.
The space $\mathcal{C}^{2}$ consists of complex functions whose $2$th-order partial derivatives are continuous.

\begin{defn}\cite{LAA}
Let $\Omega$ be a bounded domain of $\mathbb{C}^{m}, m>1$. A function $u\in \mathcal{C}^{2}(\Omega)$ is said to be pluriharmonic if it satisfies the $m^{2}$ differential equations
$\frac{\partial ^{2}u}{\partial w_{i}\partial\bar{w}_{j}}=0$ for $1\leq i,j\leq m.$
\end{defn}

\begin{defn}\cite{PL,KO}
A real-valued function $u:\Omega\rightarrow \mathbf{R}\cup\{-\infty\}$ $(u\not\equiv-\infty)$ is plurisubharmonic if it satisfies the following conditions:
\begin{itemize}
  \item [(1)] $u(z)$ is upper-semicontinuous on $\Omega$;
  \item [(2)] For any arbitrary $z_{0}\in \Omega$ and some $z_{1}\in \mathbb{C}^{m}$ determined by $z_{0}$, $u(z_{0}+\lambda z_{1})$ is subharmonic with respect to $\lambda\in\mathbb{C}$.
\end{itemize}
\end{defn}
\begin{defn}\label{def2.6}
A non-negative function $g:\mathbb{R}^{m}\rightarrow[0,+\infty)$ is called log-plurisubharmonic, if the function $\log g$ is plurisubharmonic.
\end{defn}

\begin{lem}\label{lem2.8}
Let $f$ be a pluriharmonic function on $\Omega$. Then $\log|f|$ and $|f|^{p} (0<p<\infty)$ are plurisubharmonic function on $\Omega$.
\end{lem}

\section{The similarity of commuting tuples of weighted shifts}

In this section, we mainly investigate how to describe the similarity of commuting $m$-tuples of unilateral weighted shifts.
In \cite{SH}, Shields provided a necessary and sufficient condition for the similarity of unilateral weighted shift operators by
using the weight sequences. 
In the following, we give the generalization of Shields's result on commuting tuples of unilateral weighted shifts, and prove that the similarity invariants of certain $n$-hypercontrative tuples can be characterized by the difference of the curvature of operator tuples in Cowen-Douglas class.

Let $H$ be a separable Hilbert space, and $\{\mathbf{e}_{\alpha}\}_{\alpha\in \mathbf{Z}_{+}^{m}}$ be an orthonormal basis of the Hilbert space $\mathcal{H}=\ell^{2}(\mathbf{Z}_{+}^{m},H)$ composed of functions $f$ satisfying
$\Vert f \Vert^{2}=\sum\limits_{\alpha\in \mathbf{Z}_{+}^{m}}|f(\alpha)|^{2}<\infty.$
Let $\mathbf{T}=(T_{1},\cdots,T_{m})$ and $\mathbf{S}=(S_{1},\cdots,S_{m})$ be commuting $m$-tuples of unilateral weighted shift operators on $\mathcal{H}$, and their nonzero weight sequences are $\{\lambda_{\alpha}^{(1)},\cdots,\lambda_{\alpha}^{(m)}\}_{\alpha\in \mathbf{Z}_{+}^{m}}$ and $\{\tilde{\lambda}_{\alpha}^{(1)},\cdots,\tilde{\lambda}_{\alpha}^{(m)}\}_{\alpha\in \mathbf{Z}_{+}^{m}}$ respectively. That is,
$$T_{i}\mathbf{e}_{\alpha}=\lambda_{\alpha}^{(i)}\mathbf{e}_{\alpha+e_{i}}\quad\text{and}\quad
S_{i}\mathbf{e}_{\alpha}=\widetilde{\lambda}_{\alpha}^{(i)}\mathbf{e}_{\alpha+e_{i}},$$
where $1\leq i\leq m$ and $e_{i}=(0,\cdots,0,1,0,\cdots,0)\in \mathbf{Z}_{+}^{m}$ with $1$ on the $i$th position.

The following theorem is the generalization of Shields's result on commuting tuples of unilateral weighted shifts, which were given by S. Kumar and V. S. Pilidi in \cite{Kumar} and \cite{Pi}, respectively.

\begin{thm}\label{c4.2}
Let $\mathbf{T}=(T_{1},\cdots,T_{m}),\mathbf{S}=(S_{1},\cdots,S_{m})\in \mathcal{L}(H)^m$ be tuples of unilateral weighted shift operators, and their nonzero weight sequences are $\{\lambda_{\alpha}^{(1)},\cdots,\lambda_{\alpha}^{(m)}\}_{\alpha\in \mathbf{Z}_{+}^{m}}$ and $\{\tilde{\lambda}_{\alpha}^{(1)},\cdots,\tilde{\lambda}_{\alpha}^{(m)}\}_{\alpha\in \mathbf{Z}_{+}^{m}}$ respectively. Then $\mathbf{T}$ and $\mathbf{S}$ are similar if and only if there exist positive constants $C_1$ and $C_2$ such that
$$0<C_1\leq \Bigg |\frac{\prod\limits_{k=0}^{l}\lambda_{\alpha+ke_{i}}^{(i)}}{\prod\limits_{k=0}^{l}\tilde{\lambda}_{\alpha+ke_{i}}^{(i)}} \Bigg | \leq C_2$$
for any nonnegative integer $l$, $\alpha\in \mathbf{Z}_{+}^{m}$ and $1\leq i\leq m$.
\end{thm}

The result describes the similarity of commuting $m$-tuples of unilateral weighted shifts from the perspective of weight sequence.
Next, we give an example to illustrate one of the applications of the theorem above to the similarity of $m$-tuples in $\mathcal{B}_{1}^{m}(\Omega).$
Before introducing this result, we need to prove the following lemma.

\begin{lem}\label{lem5.2}
Let $\beta=(\beta_{1},\beta_{2},\cdots,\beta_{m})\in \mathbf{Z}_{+}^{m}$ and $i$ be a nonnegative integer. If $i\leq|\beta|$, then
$$\sum\limits_{\alpha\leq\beta \atop|\alpha|=i}\frac{|\alpha|!\beta!(|\beta-\alpha|)!}{\alpha!|\beta|!(\beta-\alpha)!}=1.$$
\end{lem}
\begin{proof}
In order to prove that
$\sum\limits_{\alpha\leq\beta \atop|\alpha|=i}\frac{|\alpha|!\beta!(|\beta-\alpha|)!}{\alpha!|\beta|!(\beta-\alpha)!}=\sum\limits_{\alpha\leq\beta \atop|\alpha|=i }\frac{i!\beta!(|\beta|-i)!}{\alpha!|\beta|!(\beta-\alpha)!}=1,$
we just need to verify that $\sum\limits_{\alpha\leq\beta \atop|\alpha|=i}\frac{\beta!}{\alpha!(\beta-\alpha)!}=\frac{|\beta|!}{i!(|\beta|-i)!}.$ Since
$(1+x)^{\beta_{1}}(1+x)^{\beta_{2}}\cdots(1+x)^{\beta_{m}}=(1+x)^{\beta_{1}+\beta_{2}+\cdots+\beta_{m}}=(1+x)^{|\beta|},$
by comparing the coefficients of the above $x^{i}, i\leq|\beta|$, we have that $\sum\limits_{\alpha\leq\beta \atop|\alpha|=i}{\beta_{1}\choose \alpha_{1}}{\beta_{2}\choose \alpha_{2}}\cdots{\beta_{m}\choose \alpha_{m}}={|\beta|\choose i}$. This completes the proof.
\end{proof}

Let $\widehat{\mathcal{H}}$ be the functional  Hilbert space with reproducing kernel $K(z,w)=\sum\limits_{i=0}^{\infty}a(i)(z_{1}\overline{w}_{1}+\cdots+z_{m}\overline{w}_{m})^{i}, a(i)>0$ for $w\in\mathbb{B}^{m}$, and let $\{\widehat{\mathbf{e}_{\alpha}}\}_{\alpha\in\mathbf{Z}_{+}^{m}}$ and $\{\mathbf{e}_{\alpha}\}_{\alpha\in\mathbf{Z}_{+}^{m}}$ be the orthonormal basis of spaces $\widehat{\mathcal{H}}$ and $H_{n}^{2}$, respectively. In the above notation, we have the following theorem:

\begin{thm}\label{thm5.2}\label{3}
Let $n$ be a positive integer, and let $\mathbf{T}=(T_{1},\cdots,T_{m}),\mathbf{S}^*=(S_{1}^{*},\cdots,S_{m}^{*})\in \mathcal{B}_{1}^{m}(\mathbb{B}^{m})$ be the adjoint of the multiplication operators by coordinate functions on
$\widehat{\mathcal{H}}$ and $H_{n}^{2}$, respectively. If $\mathbf{T}$ is $n$-hypercontractive, then $\mathbf{T}$ is similar to $\mathbf{S}^*$ if and only if there exists a bounded plurisubharmonic function $\psi$ such that
  $$\mathcal{K}_{\mathbf{S}^*}(w)-\mathcal{K}_\mathbf{T}(w)=\sum \limits_{i,j=1}^{m}\frac{\partial^{2}\psi(w)}{\partial w_{i}\partial \overline{w}_{j}}dw_{i}\wedge d\overline{w}_{j},\quad w\in\mathbb{B}^{m}.$$
\end{thm}
\begin{proof}
Since $\mathbf{T},\mathbf{S}^*\in \mathcal{B}_{1}^{m}(\mathbb{B}^{m})$, there are non-vanishing holomorphic sections $\widehat{\gamma}$ and $\gamma$ of $E_{\mathbf{T}}$ and $E_{\mathbf{S}^*}$, such that
\begin{equation}\label{056}
h_{\mathbf{T}}(w)=\Vert\widehat{\gamma}(w)\Vert^{2}=\sum\limits_{\alpha\in\mathbf{Z}_{+}^{m}}\widehat{\rho}(\alpha)w^{\alpha}\overline{w}^{\alpha}\quad\text{and}\quad h_{\mathbf{S}^*}(w)=\Vert \gamma(w)\Vert^{2}=\sum\limits_{\alpha\in\mathbf{Z}_{+}^{m}}\rho(\alpha)w^{\alpha}\overline{w}^{\alpha},
\end{equation}
where $\widehat{\rho}(\alpha)=a(|\alpha|)\frac{|\alpha|!}{\alpha!}$ and $\rho(\alpha)=\frac{(n+|\alpha|-1)!}{\alpha!(n-1)!}$.
From the relation between reproducing kernel and weight sequence, we have
$$T^{*}_{i}\widehat{\mathbf{e}}_{\alpha}=\sqrt{\frac{\widehat{\rho}(\alpha)}{\widehat{\rho}(\alpha+e_{i})}}\widehat{\mathbf{e}}_{\alpha+e_{i}},\quad
\quad T_{i}\widehat{\mathbf{e}}_{\alpha}=\sqrt{\frac{\widehat{\rho}(\alpha-e_{i})}{\widehat{\rho}(\alpha)}}\widehat{\mathbf{e}}_{\alpha-e_{i}},$$
and
$$S_{i}\mathbf{e}_{\alpha}=\sqrt{\frac{\rho(\alpha)}{\rho(\alpha+e_{i})}}\mathbf{e}_{\alpha+e_{i}},\quad
\quad S^{*}_{i}\mathbf{e}_{\alpha}=\sqrt{\frac{\rho(\alpha-e_{i})}{\rho(\alpha)}}\mathbf{e}_{\alpha-e_{i}}$$
where $1\leq i\leq m$ and $e_{i}=(0,\cdots,0,1,0,\cdots,0)\in\mathbf{Z}_{+}^{m}$ with $1$ on the $i$th position.

On the one hand, if $\mathbf{T}$ and $\mathbf{S}^{*}$ are similar. By Theorem \ref{c4.2}, there are positive constants  $C_{1}$ and $C_{2}$ such that
$$0<C_{1}\leq\frac{\prod\limits_{k=0}^{l}\sqrt{\frac{\rho(\alpha+ke_{i})}{\rho(\alpha+(k+1)e_{i})}}}{\prod\limits_{k=0}^{l}\sqrt{\frac{\widehat{\rho}(\alpha+ke_{i})}{\widehat{\rho}(\alpha+(k+1)e_{i})}}}
=\frac{\sqrt{\frac{\rho(\alpha)}{ \rho(\alpha+(l+1)e_{i})}}}{\sqrt{\frac{\widehat{\rho}(\alpha)}{\widehat{\rho}(\alpha+(l+1)e_{i})}}}\leq C_{2}$$
for any nonnegative integers $l$, $1\leq i\leq m$ and $\alpha\in\mathbf{Z}_{+}^{m}$.
Then for arbitrary and fixed $\alpha=(\alpha_{1},\alpha_{2},\cdots,\alpha_{m})\in \mathbf{Z}_{+}^{m}$, if $\alpha_{i}\neq 0, 1\leq i\leq m$, we obtain
$$\begin{cases}
0<C^{2}_{1}\leq\frac{\frac{\rho(\alpha-\alpha_{1}e_{1})}{\rho(\alpha)}}{\frac{\widehat{\rho}(\alpha-\alpha_{1}e_{1})}{\widehat{\rho}(\alpha)}}\leq C^{2}_{2},\\
0<C^{2}_{1}\leq\frac{\frac{\rho(\alpha-\alpha_{1}e_{1}-\alpha_{2}e_{2})}{\rho(\alpha-\alpha_{1}e_{1})}}{\frac{\widehat{\rho}(\alpha-\alpha_{1}e_{1}-\alpha_{2}e_{2})}{\widehat{\rho}(\alpha-\alpha_{1}e_{1})}}\leq C^{2}_{2},\\
\cdots\cdots\\
0<C^{2}_{1}\leq\frac{\frac{\rho(\theta)}{\rho(\alpha-\alpha_{1}e_{1}-\cdots-\alpha_{m-1}e_{m-1})}}{\frac{\widehat{\rho}(\theta)}{\widehat{\rho}(\alpha-\alpha_{1}e_{1}-\cdots-\alpha_{m-1}e_{m-1})}}\leq C^{2}_{2}.
\end{cases}$$
Further, we have
$0<C^{2m}_{1}\leq\frac{\rho(\theta)}{\widehat{\rho}(\theta)}\frac{\widehat{\rho}(\alpha)}{\rho(\alpha)}\leq C^{2m}_{2}$. Otherwise, there is $1\leq p< m$ such that $0<C^{2p}_{1}\leq\frac{\rho(\theta)}{\widehat{\rho}(\theta)}\frac{\widehat{\rho}(\alpha)}{\rho(\alpha)}\leq C^{2p}_{2}$.
For convenience, we write $r$ as $m$ or $p$ in the above two cases.
From $\rho(\theta)=\frac{(n+|\theta|-1)!}{\theta!(n-1)!}=1$, we obtain
$$0<C^{2r}_{1}\widehat{\rho}(\theta)\rho(\alpha)\leq \widehat{\rho}(\alpha)\leq C^{2r}_{2}\widehat{\rho}(\theta)\rho(\alpha),\quad \alpha\in \mathbf{Z}_{+}^{m}.$$
So
$$0<C^{2r}_{1}\widehat{\rho}(\theta)\sum\limits_{\alpha\in\mathbf{Z}_{+}^{m}}\rho(\alpha)w^{\alpha}\overline{w}^{\alpha}\leq \sum\limits_{\alpha\in\mathbf{Z}_{+}^{m}}\widehat{\rho}(\alpha)w^{\alpha}\overline{w}^{\alpha}\leq C^{2r}_{2}\widehat{\rho}(\theta)\sum\limits_{\alpha\in\mathbf{Z}_{+}^{m}}\rho(\alpha)w^{\alpha}\overline{w}^{\alpha}.$$
Let $M_1 := C^{2r}_{1}\widehat{\rho}(\theta)$ and $M_2 := C^{2r}_{2}\widehat{\rho}(\theta)$. By (\ref{056}), we have
$$0<M_1\leq\frac{h_{\mathbf{T}}(w)}{h_{\mathbf{S}^{*}}(w)}=\frac{\Vert\widehat{\gamma}(w)\Vert^{2}}{\Vert\gamma(w)\Vert^{2}}\leq M_2$$
and
$\log M_1\leq\log\frac{\Vert\widehat{\gamma}(w)\Vert^{2}}{\Vert\gamma(w)\Vert^{2}}\leq\log M_2$ for $ w\in\mathbb{B}^{m}.$
This means that $\log\Arrowvert\frac{\widehat{\gamma}(w)}{\gamma(w)}\Arrowvert^{2}$ is a bounded function. From $\widehat{\gamma}(w)$ and $\gamma(w)$ are holomorphic,
we get $\frac{\widehat{\gamma}(w)}{\gamma(w)}$ is also holomorphic. By Lemma \ref{lem2.8}, we know that $\log\Vert\frac{\widehat{\gamma}(w)}{\gamma(w)}\Vert$ is a bounded plurisubharmonic function,
and so is $\log\Vert\frac{\widehat{\gamma}(w)}{\gamma(w)}\Vert^{2}=2\log\Vert\frac{\widehat{\gamma}(w)}{\gamma(w)}\Vert$. Thus
\begin{eqnarray}\nonumber
\mathcal{K}_{\mathbf{S}^*}(w)-\mathcal{K}_\mathbf{T}(w)
&=&-\sum \limits_{i,j=1}^{m}\frac{\partial^{2}\log\Vert\gamma(w)\Vert^{2}}{\partial w_{i}\partial \overline{w}_{j}}dw_{i}\wedge d\overline{w}_{j}+
\sum \limits_{i,j=1}^{m}\frac{\partial^{2}\log\Vert\widehat{\gamma}(w)\Vert^{2}}{\partial w_{i}\partial \overline{w}_{j}}dw_{i}\wedge d\overline{w}_{j}\\[4pt]\nonumber
&=&\sum \limits_{i,j=1}^{m}\frac{\partial^{2}\log\Vert\frac{\widehat{\gamma}(w)}{\gamma(w)}\Vert^{2}}{\partial w_{i}\partial \overline{w}_{j}}dw_{i}\wedge d\overline{w}_{j},\nonumber
\end{eqnarray}
where $\log\Vert\frac{\widehat{\gamma}(w)}{\gamma(w)}\Vert^{2}$ is a bounded plurisubharmonic function.

On the other hand, if there is a bounded plurisubharmonic function $\psi$ such that
  $$\mathcal{K}_{\mathbf{S}^*}(w)-\mathcal{K}_\mathbf{T}(w)=\sum \limits_{i,j=1}^{m}\frac{\partial^{2}\psi(w)}{\partial w_{i}\partial \overline{w}_{j}}dw_{i}\wedge d\overline{w}_{j},\quad w\in\mathbb{B}^{m}.$$
It follows that
$$\frac{\partial^{2}\psi(w)}{\partial w_{i}\partial \overline{w}_{j}}=\frac{\partial^{2}\log\Vert\frac{\widehat{\gamma}(w)}{\gamma(w)}\Vert^{2}}{\partial w_{i}\partial \overline{w}_{j}},\quad $$
that is, $\frac{\partial^{2}}{\partial w_{i}\partial \overline{w}_{j}}\log\frac{\Vert\frac{\widehat{\gamma}(w)}{\gamma(w)}\Vert^{2}}{e^{\psi(w)}}=0,1\leq i,j\leq m.$ Therefore, there exists a nonzero holomorphic function $\phi$ such that $\frac{\Vert\frac{\widehat{\gamma}(w)}{\gamma(w)}\Vert^{2}}{e^{|\psi(w)|}}=|\phi(w)|^{2}.$
and $\Vert\frac{\widehat{\gamma}(w)}{\phi(w)\gamma(w)}\Vert^{2}=e^{|\psi(w)|}$. Since $\psi$ is bounded, there are constants $C_{3}$ and $C_{4}$ such that
$C_{3}\leq|\psi(w)|\leq C_{4}.$
Set $\widetilde{m}:=e^{C_{3}}$ and $\widetilde{M}:=e^{C_{4}}$, then
\begin{equation}\label{051}
0<\widetilde{m}\leq\frac{\Vert\widehat{\gamma}(w)\Vert^{2}}{\Vert\phi(w)\gamma(w)\Vert^{2}}\leq\widetilde{ M},\quad w\in\mathbb{B}^{m}.
\end{equation}
Note that for all $\alpha\in \mathbf{Z}_{+}^{m}$,
\begin{equation}\label{058}\nonumber
\mathbf{T}^{*\alpha}\mathbf{T}^{\alpha}\mathbf{e}_{\beta}=
\begin{cases}
\widehat{\rho}(\beta)^{-1}\widehat{\rho}(\beta-\alpha)\widehat{\mathbf{e}}_{\beta}\quad \text{if}\,\,\alpha\leq\beta,\\
0\quad\quad\quad\quad\quad\quad\quad\quad\,\, \text{if}\,\,\alpha\nleq\beta.
\end{cases}
\end{equation}
Since
$\mathbf{T}=(T_{1},\cdots,T_{m})$ is an $n$-hypercontraction and $\widehat{\rho}(\alpha)=a(|\alpha|)\frac{|\alpha|!}{\alpha!}$, by Lemma \ref{lem5.2}, we have for $1\leq k\leq n$ and $\beta\in\mathbf{Z}_{+}^{m}$,
\begin{eqnarray}\nonumber
\langle\triangle_{\mathbf{T}}^{(k)}\widehat{\mathbf{e}}_{\beta},\widehat{\mathbf{e}}_{\beta}\rangle&=&\sum \limits_{|\alpha|\leq k}(-1)^{|\alpha|}\frac{k!}{\alpha!(k-|\alpha|)!}\langle\mathbf{T}^{\alpha}\widehat{\mathbf{e}}_{\beta},\mathbf{T}^{\alpha}\widehat{\mathbf{e}}_{\beta}\rangle\\[4pt]\nonumber
&=&\sum \limits_{|\alpha|\leq k \atop \alpha\leq\beta}(-1)^{|\alpha|}\frac{k!}{\alpha!(k-|\alpha|)!}\frac{\widehat{\rho}(\beta-\alpha)}{\widehat{\rho}(\beta)}\\[4pt]\nonumber
&=&\sum \limits_{|\alpha|\leq k \atop \alpha\leq\beta}(-1)^{|\alpha|}\frac{k!}{\alpha!(k-|\alpha|)!}\frac{a(|\beta-\alpha|)|\beta-\alpha|!}{(\beta-\alpha)!}\frac{\beta!}{a(|\beta|)|\beta|!}\\[4pt]\nonumber
&=&\frac{1}{a(|\beta|)}\sum \limits_{|\alpha|\leq k\atop \alpha\leq\beta}(-1)^{|\alpha|}\frac{k!}{|\alpha|!(k-|\alpha|)!}a(|\beta|-|\alpha|)\frac{|\beta-\alpha|!}{(\beta-\alpha)!}\frac{|\alpha|!\beta!}{\alpha!|\beta|!}\\[4pt]\nonumber
&=&\frac{1}{a(|\beta|)}\sum \limits_{i=0}^{k}[(-1)^{i}\frac{k!}{i!(k-i)!}a(|\beta|-i)\sum\limits_{|\alpha|=i \atop \alpha\leq\beta}\frac{(|\beta|-i)!i!\beta!}{(\beta-\alpha)!\alpha!|\beta|!}]\\[4pt]\nonumber
&=&\frac{1}{a(|\beta|)}\sum \limits_{i=0}^{k}(-1)^{i}\frac{k!}{i!(k-i)!}a(|\beta|-i)\\[4pt]\nonumber
&\geq&0,\nonumber
\end{eqnarray}
then
\begin{equation}\label{052}
\sum \limits_{i=0}^{k}(-1)^{i}\frac{k!}{i!(k-i)!}a(|\beta|-i)=\sum \limits_{i=0}^{k}(-1)^{i}{k \choose i}a(|\beta|-i)\geq0,\,\,|\beta|\geq k.
\end{equation}
Note that
\begin{eqnarray}\nonumber
&&\frac{h_{\mathbf{T}}(w)}{h_{\mathbf{S}^{*}}(w)}\\[4pt]\nonumber
&=&(1-|w|^{2})^{n}\sum\limits_{i=0}^{\infty}a(i)(|w|^{2})^{i}\\[4pt]\nonumber
&=&a(0)+\sum\limits_{j=0}^{1}(-1)^{j}{n \choose j}a(1-j)|w|^{2}+\cdots
+\sum\limits_{j=0}^{l}(-1)^{j}{n \choose j}a(l-j)|w|^{2l}+\cdots+\sum\limits_{j=0}^{n}(-1)^{j}{n \choose j}a(n-j)|w|^{2n}\\[4pt]\nonumber
& &+\sum\limits_{j=0}^{n}(-1)^{j}{n \choose j}a(n+1-j)|w|^{2(n+1)}+\cdots+\sum\limits_{j=0}^{n}(-1)^{j}{n \choose j}a(k-j)|w|^{2k}+\cdots.\nonumber
\end{eqnarray}
From the known fact ${n \choose i}={n-1 \choose i}+{n-1 \choose i-1}$, we obtain
\begin{equation}\label{053}
\sum\limits_{i=0}^{m}(-1)^{i}{n \choose i}=(-1)^{m}{n-1 \choose m}.
\end{equation}
By equations (\ref{052}) and (\ref{053}), we have that
\begin{eqnarray}\label{054}
&&\max\limits_{|w|\leq1}\frac{h_{\mathbf{T}}(w)}{h_{\mathbf{S}^{*}}(w)}\\[4pt]\nonumber
&=&\sum\limits_{l=0}^{n}[\sum\limits_{j=0}^{l}(-1)^{j}{n \choose j}a(l-j)]+\sum\limits_{l=n+1}^{\infty}[\sum\limits_{j=0}^{n}(-1)^{j}{n \choose j}a(l-j)]\\[4pt]\nonumber
&=&\lim\limits_{k\rightarrow\infty}[a(k)+(\sum\limits_{i=0}^{1}(-1)^{i}{n \choose i})a(k-1)+\cdots+(\sum\limits_{i=0}^{n-1}(-1)^{i}{n \choose i})a(k-n+1)]\\[4pt]\nonumber
&=&\lim\limits_{k\rightarrow\infty}\sum\limits_{i=0}^{n-1}(-1)^{i}{n-1 \choose i}a(k-i).\nonumber
\end{eqnarray}
From (\ref{051}) and (\ref{054}), there are positive constants $m':=a(0)$ and $M'$ such that
\begin{equation}\label{055}
0<m'\leq\lim\limits_{k\rightarrow\infty}\sum\limits_{i=0}^{n-1}(-1)^{i}{n-1 \choose i}a(k-i)\leq M',
\end{equation}
and then $$0<m'\leq\frac{h_{\mathbf{T}}(w)}{h_{\mathbf{S}^{*}}(w)}=\frac{\Vert\widehat{\gamma}(w)\Vert^{2}}{\Vert\gamma(w)\Vert^{2}}\leq M',\quad w\in\mathbb{B}^{m}.$$ Otherwise, we have
$\frac{\Vert\widehat{\gamma}(w)\Vert^{2}}{\Vert\gamma(w)\Vert^{2}}\rightarrow \infty$ as $|w|\rightarrow 1,$ and from (\ref{051}) we know that $\frac{1}{\phi(w)}\rightarrow0$ as $|w|\rightarrow 1.$ By the maximum modulus theorem of holomorphic functions, we have that $\frac{1}{\phi(w)}=0$ for any $w\in\mathbb{B}^{m}$. This contradicts that $\phi$ is nonzero holomorphic.

{\bf Claim:} For any positive integer $n$, and any positive sequence $\{a(k)\}_{k\geq n-1}$. If $n$ and $\{a(k)\}_{k\geq n-1}$ satisfy inequality (\ref{055}), then there exists non-zero constants $c_{1}$ and $c_{2}$ such that $$0<c_{1}\leq\lim\limits_{k\rightarrow\infty}\frac{a(k)}{k^{n-1}}\leq c_{2},$$ denoted by $a(k)=O(k^{n-1})(k\rightarrow\infty)$.

In the case of $n=1$, it is obviously true.

When $n=2$, from (\ref{055}), we know that $m'\leq a(k)-a(k-1)\leq M'$. Furthermore, $$km'+a(0)\leq a(k)\leq kM'+a(0),$$
which means that $a(k)=O(k)(k\rightarrow\infty)$ holds.

Assume that the Claim is valid for $n=t$. That is, $a(k)=O(k^{t-1})(k\rightarrow\infty)$. Now we prove that $n=t+1$ is also true.
Set $u(k):=
\begin{cases}
a(k)-a(k-1)\quad k\geq1,\\
a(0)\quad\quad\quad\quad\quad\quad k=0.
\end{cases}$
Since $\mathbf{T}$ is a contraction, the positive sequence $\{a(k)\}_{k\geq0}$ is increasing, $u(k)\geq0$ for any nonnegative integer $k$. By (\ref{055}) and ${t \choose i}={t-1 \choose i}+{t-1 \choose i-1}$, we have
\begin{equation}\nonumber
0<m'\leq\lim\limits_{k\rightarrow\infty}\sum\limits_{i=0}^{t}(-1)^{i}{t \choose i}a(k-i)=\lim\limits_{k\rightarrow\infty}\sum\limits_{i=0}^{t-1}(-1)^{i}{t-1 \choose i}u(k-i)\leq M'.
\end{equation}
Therefore, $u(k)=O(k^{t-1})(k\rightarrow\infty)$. That is, there exist $m_{1}$ and $M_{1}$ such that
$$0<m_{1}k^{t-1}\leq u(k)= a(k)-a(k-1)\leq M_{1}k^{t-1}.$$ So
$$m_{1}\sum\limits_{i=1}^{k}i^{t-1}+a(0)\leq a(k)\leq M_{1}\sum\limits_{i=1}^{k}i^{t-1}+a(0).$$
Set $l_{t-1}:=\sum\limits_{i=1}^{k}i^{t-1}.$
It is well-known that
$l_{t-1}=\frac{1}{t}\Big[(1+k)^{t}-(1+k)-(\sum\limits_{i=2}^{t-1}{t \choose i}l_{t-i})\Big].$
Then $$\sum\limits_{i=1}^{k}i^{t-1}=O(k^{t})(k\rightarrow\infty)\quad\text{and}\quad a(k)=O(k^{t})(k\rightarrow\infty).$$ This proves that the Claim is true.

By the Claim and equations (\ref{055}), $\widehat{\rho}(\alpha)=a(|\alpha|)\frac{|\alpha|!}{\alpha!}$, we have  $\widehat{\rho}(\alpha)=O(|\alpha|^{n-1}\frac{|\alpha|!}{\alpha!})(|\alpha|\rightarrow\infty).$
Thus, for any $\alpha\in \mathbf{Z}_{+}^{m}$, $1\leq i\leq m$ and any positive integer $l$,
\begin{eqnarray}\nonumber
\frac{\prod\limits_{k=0}^{l}\sqrt{\frac{\rho(\alpha+ke_{i})}{\rho(\alpha+(k+1)e_{i})}}}{\prod\limits_{k=0}^{l}\sqrt{\frac{\widehat{\rho}(\alpha+ke_{i})}{\widehat{\rho}(\alpha+(k+1)e_{i})}}}&=
&\sqrt{\frac{\rho(\alpha)\widehat{\rho}(\alpha+(l+1)e_{i})}{\widehat{\rho}(\alpha)\rho(\alpha+(l+1)e_{i})}}\\[4pt]\nonumber
&=&O\Biggl(\sqrt{\frac{(n+|\alpha|-1)!}{(n+|\alpha|+l)!}\frac{|\alpha+(l+1)e_{i}|^{n-1}}{|\alpha|^{n-1}}\frac{|\alpha+(l+1)e_{i}|!}{|\alpha|!}}\Biggl)\quad |\alpha|\rightarrow\infty.\nonumber
\end{eqnarray}
Note that for any positive integer $l,$
$$\frac{(n+|\alpha|-1)!}{(n+|\alpha|+l)!}\frac{|\alpha+(l+1)e_{i}|^{n-1}}{|\alpha|^{n-1}}\frac{|\alpha+(l+1)e_{i}|!}{|\alpha|!}\longrightarrow 1\quad |\alpha|\rightarrow\infty.$$
Therefore, for any positive integer $l,$ there is an integer $N_{1}$, so that when $|\alpha|\geq N_{1}$,
$$\frac{1}{2}\leq\frac{(n+|\alpha|-1)!}{(n+|\alpha|+l)!}\frac{|\alpha+(l+1)e_{i}|^{n-1}}{|\alpha|^{n-1}}\frac{|\alpha+(l+1)e_{i}|!}{|\alpha|!}\leq\frac{3}{2}.$$
Then there are positive constants $m_{2}$ and $M_{2}$ such that
$$0<\sqrt{\frac{1}{2}}m_{2}\leq\frac{\prod\limits_{k=0}^{l}\sqrt{\frac{\rho(\alpha+ke_{i})}{\rho(\alpha+(k+1)e_{i})}}}{\prod\limits_{k=0}^{l}\sqrt{\frac{\widehat{\rho}(\alpha+ke_{i})}{\widehat{\rho}(\alpha+(k+1)e_{i})}}}\leq\sqrt{\frac{3}{2}}M_{2},\quad|\alpha|\geq N_{1},l\geq0.$$
When $|\alpha|< N_{1},$ we have
\begin{eqnarray}\nonumber
\frac{\prod\limits_{k=0}^{l}\sqrt{\frac{\rho(\alpha+ke_{i})}{\rho(\alpha+(k+1)e_{i})}}}{\prod\limits_{k=0}^{l}\sqrt{\frac{\widehat{\rho}(\alpha+ke_{i})}{\widehat{\rho}(\alpha+(k+1)e_{i})}}}&
=&O\Biggl(\sqrt{\frac{\rho(\alpha)}{\widehat{\rho}(\alpha)}}
\sqrt{\frac{(|\alpha|+l+1)^{n-1}(n-1)!}{(|\alpha|+l+2)(|\alpha|+l+3)\cdots(|\alpha|+l+n)}}\,\Biggl)\quad l\rightarrow\infty\nonumber
\end{eqnarray}
and $\frac{(|\alpha|+l+1)^{n-1}}{(|\alpha|+l+2)(|\alpha|+l+3)\cdots(|\alpha|+l+n)}=O(1) (l\rightarrow\infty).$
Let $$m_{3}:=\min\Biggl\{\sqrt{\frac{(n-1)!\rho(\alpha)}{\widehat{\rho}(\alpha)}}:|\alpha|< N_{1}, \alpha\in \mathbf{Z}_{+}^{m}\Biggl\}\text{ and } M_{3}:=\max\Biggl\{\sqrt{\frac{(n-1)!\rho(\alpha)}{\widehat{\rho}(\alpha)}}:|\alpha|< N_{1}, \alpha\in \mathbf{Z}_{+}^{m}\Biggl\}. $$
Then there are $\varepsilon_{0}>0$ and positive integer $N_{2}$ such that when $|\alpha|< N_{1}$ and $l\geq N_{2}$, there is
$$0<m_{3}(1-\varepsilon_{0})\leq\sqrt{\frac{\rho(\alpha)}{\widehat{\rho}(\alpha)}}
\sqrt{\frac{(|\alpha|+l+1)^{n-1}(n-1)!}{(|\alpha|+l+2)(|\alpha|+l+3)\cdots(|\alpha|+l+n)}}\leq M_{3}(1+\varepsilon_{0}).$$
Then there are positive constants $m_{4}$ and $M_{4}$ such that
$$0<m_{3}(1-\varepsilon_{0})m_{4}\leq\frac{\prod\limits_{k=0}^{l}\sqrt{\frac{\rho(\alpha+ke_{i})}{\rho(\alpha+(k+1)e_{i})}}}{\prod\limits_{k=0}^{l}\sqrt{\frac{\widehat{\rho}(\alpha+ke_{i})}{\widehat{\rho}(\alpha+(k+1)e_{i})}}}\leq M_{3}(1+\varepsilon_{0})M_{4},\quad|\alpha|< N_{1},l\geq N_{2}.$$
Set
$$m_{5}:=\min\Biggl\{\frac{\sqrt{\frac{\rho(\alpha)}{\rho(\alpha+(l+1)e_{i})}}}{\sqrt{\frac{\widehat{\rho}(\alpha)}{\widehat{\rho}(\alpha+(l+1)e_{i})}}}:|\alpha|< N_{1}, l<N_{2}\Biggl\}\quad \text{and}\quad M_{5}:=\max\Biggl\{\frac{\sqrt{\frac{\rho(\alpha)}{\rho(\alpha+(l+1)e_{i})}}}{\sqrt{\frac{\widehat{\rho}(\alpha)}{\widehat{\rho}(\alpha+(l+1)e_{i})}}}:|\alpha|< N_{1}, l<N_{2}\Biggl\}.$$
Then there are  positive constants
$$C'_{1}=\min\biggl\{\sqrt{\frac{1}{2}}m_{2},m_{3}(1-\varepsilon_{0})m_{4},m_{5}\biggl\}\quad\text{and}\quad
C'_{2}=\max\biggl\{\sqrt{\frac{3}{2}}M_{2},M_{3}(1+\varepsilon_{0})M_{4}, M_{5}\biggl\}$$ such that
$$0<C'_{1}\leq\frac{\prod\limits_{k=0}^{l}\sqrt{\frac{\rho(\alpha+ke_{i})}{\rho(\alpha+(k+1)e_{i})}}}{\prod\limits_{k=0}^{l}\sqrt{\frac{\widehat{\rho}(\alpha+ke_{i})}{\widehat{\rho}(\alpha+(k+1)e_{i})}}}
\leq C'_{2}$$
for any nonnegative integer $l$, $\alpha\in \mathbf{Z}_{+}^{m}$ and $1\leq i\leq m$.
By Theorem \ref{c4.2}, $\mathbf{T}$ is similar to $\mathbf{S}^*$.
\end{proof}

\section{The $N$-hypercontraction of commuting $m$-tuples of backward weighted shifts }
In this section, we mainly consider the description of  $n$-hypercontraction operator tuples. In the following, we give a necessary condition that commuting $m$-tuples of backward weighted shifts are $n$-hypercontractive, and use some specific examples to illustrate that $n$-hypercontraction in theorem \ref{3} is indispensable.

Before introducing the main theorem in this section, we need to prove the following lemma.
\begin{lem}\label{lem1}
Let $n\geq2$ be an integer. For integer $k,l$, if $2\leq k\leq n$ and $l>n$, then the following equations hold:
\begin{itemize}
  \item [(1)]$\sum\limits_{i=0}^{k}(-1)^{-i}{n-2+i \choose i}{n \choose k-i}=0$ and $\sum\limits_{i=0}^{k}(-1)^{-i}{n-2+i \choose i}{n \choose k-i}i=0$.
  \item [(2)]$\sum\limits_{i=l-n}^{l}
(-1)^{-i}{n-2+i \choose i}{n \choose l-i}=0$ and $\sum\limits_{i=l-n}^{l}
(-1)^{-i}{n-2+i \choose i}{n \choose l-i}i=0.$
\end{itemize}
\end{lem}
\begin{proof}
Since $(1-x)^{n}=\sum\limits_{i=0}^{n}{n \choose i}(-x)^{i}$ and $\frac{1}{(1-x)^{n}}=\sum\limits_{i=0}^{\infty}{n-1+i \choose i}x^{i}$ for $|x|<1$, we have that
\begin{eqnarray*}
1-x&=&(1-x)^{n}\frac{1}{(1-x)^{n-1}}\\
&=&\Big[\sum\limits_{j=0}^{n}{n \choose j}(-x)^{j}\Big]\Big[\sum\limits_{i=0}^{\infty}{n-2+i \choose i}x^{i}\Big]\\
&=&\sum\limits_{k=0}^{n}\Big[\sum\limits_{i=0}^{k}(-1)^{k-i}{n-2+i \choose i}{n \choose k-i}\Big]x^{k}+\sum\limits_{k=n+1}^{\infty}\Big[\sum\limits_{i=k-n}^{k}(-1)^{k-i}{n-2+i \choose i}{n \choose k-i}\Big]x^{k}
\end{eqnarray*}
and
\begin{eqnarray*}
n-1&=&(1-x)^{n}\Big(\frac{1}{(1-x)^{n-1}}\Big)^{\prime}\\
&=&\Big[\sum\limits_{j=0}^{n}{n \choose j}(-x)^{j}\Big]\Big[\sum\limits_{i=1}^{\infty}{n-2+i \choose i}ix^{i-1}\Big]\\
&=&\sum\limits_{k=1}^{n}\Big[\sum\limits_{i=0}^{k}(-1)^{k-i}i{n-2+i \choose i}{n \choose k-i}\Big]x^{k-1}+\sum\limits_{k=n+1}^{\infty}\Big[\sum\limits_{i=k-n}^{k}(-1)^{k-i}i{n-2+i \choose i}{n \choose k-i}\Big]x^{k-1}.
\end{eqnarray*}
It can be proved by comparing the coefficients of $x^{k}$ on the left and right sides of the above formulas.
\end{proof}

\begin{thm}\label{thm2}
Let $m\geq2$ be a positive integer and $\mathbf{T}=(T_{1},\cdots,T_{m})$ be a commuting $m$-tuple of backward weighted shift operators on Hilbert space $\mathcal{H}$ with reproducing kernel
$K(z,w)=\sum\limits_{\alpha\in\mathbf{Z}_{+}^{m}}\rho(\alpha)z^{\alpha}\overline{w}^{\alpha}$. If $\mathbf{T}$ is $n$-hypercontractive, then for any non-zero
$\alpha=(\alpha_{1},\cdots,\alpha_{m})\in\mathbf{Z}_{+}^{m}$, $$\sum\limits_{\mbox{\tiny$\begin{array}{c}
 \beta\in{\mathbf{Z}}_{+}^{m}\\
\beta\leq\alpha \\
|\alpha-\beta|=1\end{array}$}}\frac{\rho(\beta)}{\rho(\alpha)}\leq\frac{|\alpha|}{|\alpha|+n-1}.$$
\end{thm}
\begin{proof}
Let $\{\mathbf{e}_{\alpha}\}_{\alpha\in \mathbf{Z}_{+}^{m}}$ be an orthonormal basis of the Hilbert space $\ell^{2}(\mathbf{Z}_{+}^{m},\mathcal{H})$. If $\mathbf{T}$ is $n$-hypercontractive, we know that $\triangle_{\mathbf{T}}^{(k)}\geq0$ for all $1\leq k\leq n$.
It follows that for any $\alpha\in\mathbf{\mathbf{Z}}_{+}^{m}$ and $1\leq k\leq n$,
$$\triangle_{\mathbf{T}}^{(k)}\mathbf{e}_{\alpha}=\sum \limits_{\beta\in {\mathbf{Z}}_{+}^{m}\atop|\beta|\leq k}(-1)^{|\beta|}\frac{k!}{\beta!(k-|\beta|)!}\mathbf{T}^{*\beta}\mathbf{T}^{\beta}\mathbf{e}_{\alpha}
=\sum_{\mbox{\tiny$\begin{array}{c}
  \beta\in {\mathbf{Z}}_{+}^{m}\\
 \beta\leq\alpha\\
|\beta|\leq k\end{array}$}}(-1)^{|\beta|}\frac{k!}{\beta!(k-|\beta|)!}\frac{\rho(\alpha-\beta)}{\rho(\alpha)}\mathbf{e}_{\alpha}\geq0,$$ that is,
\begin{equation}\label{eq2}
\sum_{\mbox{\tiny$\begin{array}{c}
  \beta\in {\mathbf{Z}}_{+}^{m}\\
 \beta\leq\alpha\\
|\beta|\leq k\end{array}$}}(-1)^{|\beta|}\frac{k!}{\beta!(k-|\beta|)!}\frac{\rho(\alpha-\beta)}{\rho(\alpha)}\geq0.
\end{equation}
We will prove that for any $\alpha=(\alpha_{1},\cdots,\alpha_{m})\in\mathbf{Z}_{+}^{m}$,
\begin{eqnarray}\label{eq1}
&&\sum_{\mbox{\tiny$\begin{array}{c}
  \beta\in {\mathbf{Z}}_{+}^{m}\\
 \beta\leq\alpha\\
|\beta|\leq n\end{array}$}}(-1)^{|\beta|}\frac{n!}{\beta!(n-|\beta|)!}\frac{\rho(\alpha-\beta)}{\rho(\alpha)}\\
&=&1-\sum\limits_{j=1}^{k}(n+x_{e_{i_{j}}})\frac{\rho(\alpha-e_{i_{j}})}{\rho(\alpha)}+\sum\limits_{\gamma\in{\mathbf{Z}}_{+}^{m}\atop \theta<\gamma<\alpha}x_{\gamma}
[\sum_{\mbox{\tiny$\begin{array}{c}
 \xi\in {\mathbf{Z}}_{+}^{m}\\
 \xi\leq\alpha-\gamma\\
|\xi|\leq n\end{array}$}}(-1)^{|\xi|}\frac{n!}{\xi!(n-|\xi|)!}\frac{\rho(\alpha-\gamma-\xi)}{\rho(\alpha-\gamma)}]\frac{\rho(\alpha-\gamma)}{\rho(\alpha)}\nonumber
\end{eqnarray}
where $x_{\gamma}=-\frac{(n-2+|\gamma|)!}{\gamma!(n-2)!}\frac{|\alpha-\gamma|}{|\alpha|}$ for $\theta<\gamma<\alpha.$

Note that
\begin{eqnarray}
&&\sum\limits_{\gamma\in{\mathbf{Z}}_{+}^{m}\atop \theta<\gamma<\alpha}x_{\gamma}
[\sum\limits_{\mbox{\tiny$\begin{array}{c}
 \xi\in {\mathbf{Z}}_{+}^{m}\\
 \xi\leq\alpha-\gamma\\
|\xi|\leq n\end{array}$}}(-1)^{|\xi|}\frac{n!}{\xi!(n-|\xi|)!}\frac{\rho(\alpha-\gamma-\xi)}{\rho(\alpha-\gamma)}]\frac{\rho(\alpha-\gamma)}{\rho(\alpha)}\\ \nonumber
&=&\sum\limits_{\gamma\in{\mathbf{Z}}_{+}^{m}\atop \theta<\gamma<\alpha}
\sum\limits_{\mbox{\tiny$\begin{array}{c}
 \xi\in {\mathbf{Z}}_{+}^{m}\\
 \gamma\leq\xi+\gamma\leq\alpha\\
|\xi|\leq n\end{array}$}}(-1)^{|\xi|}x_{\gamma}\frac{n!}{\xi!(n-|\xi|)!}\frac{\rho(\alpha-(\gamma+\xi))}{\rho(\alpha)}\\[4pt]\nonumber
&=&\sum\limits_{\gamma\in{\mathbf{Z}}_{+}^{m}\atop \theta<\gamma<\alpha}
\sum\limits_{\mbox{\tiny$\begin{array}{c}
 \beta\in {\mathbf{Z}}_{+}^{m}\\
 \gamma\leq\beta\leq\alpha\\
|\beta-\gamma|\leq n\end{array}$}}(-1)^{|\beta|-|\gamma|+1}\frac{(n-2+|\gamma|)!}{\gamma!(n-2)!}\frac{|\alpha-\gamma|}{|\alpha|}\frac{n!}{(\beta-\gamma)!(n-|\beta-\gamma|)!}\frac{\rho(\alpha-\beta)}{\rho(\alpha)}.\nonumber
\end{eqnarray}
We will prove equation (\ref{eq1}) in the following three cases. 

$\textbf{Case 1:}$
When $|\beta|= 0,1$. It is easy to see that the coefficients of $\frac{\rho(\alpha-\beta)}{\rho(\alpha)}$ on both sides of equation (\ref{eq1}) are equal.

$\textbf{Case 2:}$
When $2\leq|\beta|\leq n,$ from Lemma \ref{lem5.2} and Lemma \ref{lem1}, we have that
\begin{eqnarray*}
&&\sum\limits_{\gamma\in{\mathbf{Z}}_{+}^{m}\atop \theta<\gamma\leq\beta}(-1)^{|\beta|-|\gamma|+1}\frac{(n-2+|\gamma|)!}{\gamma!(n-2)!}\frac{|\alpha-\gamma|}{|\alpha|}\frac{n!}{(\beta-\gamma)!(n-|\beta-\gamma|)!}-(-1)^{|\beta|}\frac{n!}{\beta!(n-|\beta|)!}\\
&=&\sum\limits_{\gamma\in{\mathbf{Z}}_{+}^{m}\atop \gamma\leq\beta}(-1)^{|\beta|-|\gamma|+1}\frac{(n-2+|\gamma|)!}{\gamma!(n-2)!}\frac{|\alpha-\gamma|}{|\alpha|}\frac{n!}{(\beta-\gamma)!(n-|\beta-\gamma|)!}\\
&=&\sum\limits_{i=0}^{|\beta|}
(-1)^{|\beta|-i+1}\frac{(n-2+i)!}{(n-2)!}\frac{|\alpha|-i}{|\alpha|}\frac{n!}{(n-|\beta|+i)!}(\sum\limits_{\mbox{\tiny$\begin{array}{c}
 \gamma\in{\mathbf{Z}}_{+}^{m}\\
 \gamma\leq\beta\\
|\gamma|=i\end{array}$}}\frac{1}{\gamma!(\beta-\gamma)!})\\
&=&(-1)^{|\beta|+1}\frac{|\beta|!}{\beta!}\sum\limits_{i=0}^{|\beta|}
(-1)^{-i}\frac{(n-2+i)!}{i!(n-2)!}\frac{n!}{(|\beta|-i)!(n-|\beta|+i)!}\frac{|\alpha|-i}{|\alpha|}\\
&=&(-1)^{|\beta|+1}\frac{|\beta|!}{\beta!}\Big\{\Big[\sum\limits_{i=0}^{|\beta|}
(-1)^{-i}{n-2+i \choose i}{n \choose |\beta|-i}\Big]- \frac{1}{|\alpha|}\Big[\sum\limits_{i=0}^{|\beta|}
(-1)^{-i}{n-2+i \choose i}{n \choose |\beta|-i}i\Big]\Big\}\\
&=&0.
\end{eqnarray*}
This means that equation (\ref{eq1}) is valid in this case.

$\textbf{Case 3:}$
When $n<|\beta|\leq |\alpha|,$ from Lemma \ref{lem5.2} and Lemma \ref{lem1} again, we have that
\begin{eqnarray*}
&&\sum\limits_{\mbox{\tiny$\begin{array}{c}
 \gamma\in{\mathbf{Z}}_{+}^{m}\\
\theta<\gamma\leq\beta\\
|\beta-\gamma|\leq n\end{array}$}}(-1)^{|\beta|-|\gamma|+1}\frac{(n-2+|\gamma|)!}{\gamma!(n-2)!}\frac{|\alpha-\gamma|}{|\alpha|}\frac{n!}{(\beta-\gamma)!(n-|\beta-\gamma|)!}\\
&=&\sum\limits_{i=|\beta|-n}^{|\beta|}
(-1)^{|\beta|-i+1}\frac{(n-2+i)!}{(n-2)!}\frac{|\alpha|-i}{|\alpha|}\frac{n!}{(n-|\beta|+i)!}(\sum\limits_{\mbox{\tiny$\begin{array}{c}
 \gamma\in{\mathbf{Z}}_{+}^{m}\\
 \gamma\leq\beta\\
|\gamma|=i\end{array}$}}\frac{1}{\gamma!(\beta-\gamma)!})\\
&=&(-1)^{|\beta|+1}\frac{|\beta|!}{\beta!}\sum\limits_{i=|\beta|-n}^{|\beta|}
(-1)^{-i}\frac{(n-2+i)!}{i!(n-2)!}\frac{n!}{(|\beta|-i)!(n-|\beta|+i)!}\frac{|\alpha|-i}{|\alpha|}\\
&=&(-1)^{|\beta|+1}\frac{|\beta|!}{\beta!}\Big\{\Big[\sum\limits_{i=|\beta|-n}^{|\beta|}
(-1)^{-i}{n-2+i \choose i}{n \choose |\beta|-i}\Big]- \frac{1}{|\alpha|}\Big[\sum\limits_{i=|\beta|-n}^{|\beta|}
(-1)^{-i}{n-2+i \choose i}{n \choose |\beta|-i}i\Big]\Big\}\\
&=&0.
\end{eqnarray*}
This shows that equation (\ref{eq1}) is also valid in Case 3.

From equations (\ref{eq2}), (\ref{eq1}) and $x_{\gamma}=-\frac{(n-2+|\gamma|)!}{\gamma!(n-2)!}\frac{|\alpha-\gamma|}{|\alpha|}<0$ for $0<\gamma<\alpha$, we know that $1-\sum\limits_{j=1}^{k}(n+x_{e_{i_{j}}})\frac{\rho(\alpha-e_{i_{j}})}{\rho(\alpha)}=1-\frac{|\alpha|+n-1}{|\alpha|}\sum\limits_{j=1}^{k}\frac{\rho(\alpha-e_{i_{j}})}{\rho(\alpha)}>0$,
then  $\sum\limits_{j=1}^{k}\frac{\rho(\alpha-e_{i_{j}})}{\rho(\alpha)}\leq\frac{|\alpha|}{|\alpha|+n-1},$
where $\alpha=(\alpha_{1},\cdots,\alpha_{m})\in\mathbf{Z}_{+}^{m}$ and $\alpha_{i_{j}}\neq0$ for $1\leq j\leq k$. That is,
$\sum\limits_{\mbox{\tiny$\begin{array}{c}
 \beta\in{\mathbf{Z}}_{+}^{m}\\
\beta\leq\alpha \\
|\alpha-\beta|=1\end{array}$}}\frac{\rho(\beta)}{\rho(\alpha)}\leq\frac{|\alpha|}{|\alpha|+n-1}$.
\end{proof}

\begin{cor}
Let $\mathbf{T}=(T_{1},\cdots,T_{m})$ be the adjoint of the multiplication operators by coordinate functions on some Hilbert space with reproducing kernel $K(z,w)=\sum\limits_{i=0}^{\infty}a(i)(z_{1}\overline{w}_{1}+\cdots+z_{m}\overline{w}_{m})^{i},$ $a(i)>0$. If $\mathbf{T}$ is $n$-hypercontractive, then
$\frac{a(i-1)}{a(i)}\leq \frac{i}{i+n-1}$ for any positive integer $i$.
\end{cor}

\begin{cor}
The adjoint of commuting $m$-tuples of unilateral weighted shifts whose spectrum is contained in $\mathbb{B}^{m}$ is not subnormal.
\end{cor}

\begin{proof}
By Theorem 5.2 in \cite{AA}, we know that for a commuting $m$-tuple whose spectrum is contained in $\mathbb{B}^{m}$,
it is $n$-hypercontractive for all positive integers $n$ if and only if it is subnormal. Let $\mathbf{T}=(T_{1},\cdots,T_{m})$ be the adjoint of commuting $m$-tuple of unilateral weighted shifts on Hilbert space $\mathcal{H}$ with reproducing kernel
$K(z,w)=\sum\limits_{\alpha\in\mathbf{Z}_{+}^{m}}\rho(\alpha)z^{\alpha}\overline{w}^{\alpha}$. Assume that $\mathbf{T}=(T_{1},\cdots,T_{m})$ is subnormal.
It follows that $\mathbf{T}$ is $n$-hypercontractive for all integers $n>0$.
By Theorem \ref{thm2}, for any $\theta\neq\alpha=(\alpha_{1},\cdots,\alpha_{m})\in\mathbf{Z}_{+}^{m}$, without losing generality, we assume that
$\alpha_{i_{0}}\neq0$, $1\leq i_{0}\leq m$, then $$\frac{\rho(\alpha-e_{i_{0}})}{\rho(\alpha)}\leq\sum\limits_{\mbox{\tiny$\begin{array}{c}
 \beta\in{\mathbf{Z}}_{+}^{m}\\
\beta\leq\alpha \\
|\alpha-\beta|=1\end{array}$}}\frac{\rho(\beta)}{\rho(\alpha)}\leq\frac{|\alpha|}{|\alpha|+n-1}.$$ Since $\lim\limits_{n\rightarrow \infty} \frac{|\alpha|}{|\alpha|+n-1}=0,$ we have $\frac{\rho(\alpha-e_{i_{j}})}{\rho(\alpha)}=0.$ This is a contradiction.
\end{proof}

\begin{ex}\label{maincor}
For the adjoint $\mathbf{S}^*=(S_{1}^{*},\cdots,S_{m}^{*})$ of the multiplication operators by coordinate functions on $H_{n}^{2}$, there exists an $m$-tuple $\mathbf{T}=(T_{1},\cdots,T_{m})$ that is not an $n$-hypercontraction and a positive, bounded, real-analytic function $\varphi$ defined on $\mathbb{B}^{m}$ such that
$$\mathcal{K}_\mathbf{\mathbf{S}^*}(w)-\mathcal{K}_\mathbf{\mathbf{T}}(w)=\sum\limits_{i,j=1}^{m}\frac{\partial^{2}\varphi(w)}{\partial w_{i}\partial \overline{w}_{j}}d w_{i}\wedge d \overline{w}_{j},\quad w \in \mathbb{B}^{m}.$$
 Moreover, $\mathbf{T}$ is not similar to $\mathbf{S}^*$.
\end{ex}
\begin{proof} Let $\mathbf{T}$ be the adjoint of multiplication operators on some Hilbert space with reproducing kernel $K(z,w)=\sum\limits_{\alpha\in\mathbf{Z}_{+}^{m}}\widetilde{\rho}(\alpha)z^{\alpha}\overline{w}^{\alpha}$. Then
$\mathcal{K}_\mathbf{\mathbf{S}^*}(w)-\mathcal{K}_\mathbf{\mathbf{T}}(w)=\sum\limits_{i,j=1}^{m}\frac{\partial^{2}}{\partial w_{i}\partial \overline{w}_{j}}\log[K(w,w)(1-|w|^{2})^{n}]d w_{i}\wedge d \overline{w}_{j}.$
Note that $\{\mathbf{e}_{\alpha}(z)=\sqrt{\rho(\alpha)}z^{\alpha}\}_{\alpha\in\mathbf{Z}_{+}^{m}}$ is an orthonormal basis of $H_{n}^{2}$ and $ S_{i}^{*}\mathbf{e}_{\alpha}=\sqrt{\frac{\rho(\alpha-e_{i})}{\rho(\alpha)}}\mathbf{e}_{\alpha-e_{i}}$ for $\rho(\alpha)=\frac{(n+|\alpha|-1)!}{\alpha!(n-1)!}$. Then letting
$$\widetilde{\rho}(\alpha)=
\begin{cases}
\frac{\rho(\alpha)}{k} \qquad \alpha=\beta^{l}+ke_{1},\,\,1 \leq k \leq l,\\
\frac{\rho(\alpha)}{k}\qquad \alpha=\beta^{l}+(2l-k)e_{1},\,\,1 \leq k \leq l,\\
\rho(\alpha)\qquad otherwise,
\end{cases}$$
where $\{\beta^{l}=(\beta_{1}^{l},\cdots,\beta_{m}^{l})\}_{l=1}^{\infty}\subseteq\mathbf{Z}_{+}^{m}$ satisfying $|\beta^{l}|>\max\{\frac{n^{n}2^{3l+1}}{(n-1)!}-n,n-2\}$ and $|\beta^{l}|+2l<|\beta^{l+1}|$, then $[\widetilde{\rho}(\alpha)-\rho(\alpha)]\frac{\alpha!}{|\alpha|!}$ is only related to $|\alpha|$, not $\alpha$, and
\begin{eqnarray*}K(w,w)&=&\frac{1}{(1-|w|^{2})^{n}}+\sum\limits_{l=0}^{\infty}\Big[\sum\limits_{\alpha\in\mathbf{Z}_{+}^{m}\atop |\alpha|=l}\left(\widetilde{\rho}(\alpha)-\rho(\alpha)\right )w^{\alpha}\overline{w}^{\alpha}\Big]\\
&=&\frac{1}{(1-|w|^{2})^{n}}+\sum\limits_{l=2}^{\infty} \Big\{\sum\limits_{k=2}^{l}\left(\widetilde{\rho}(\beta^{l}+ke_{1})-\rho(\beta^{l}+ke_{1})\right)\frac{(\beta^{l}+ke_{1})!}{|\beta^{l}+ke_{1}|!}|w|^{2|\beta^{l}+ke_{1}|}\\
&&+\sum\limits_{k=2}^{l-1}\left(\widetilde{\rho}(\beta^{l}+(2l-k)e_{1})-\rho(\beta^{l}+(2l-k)e_{1})\right)\frac{(\beta^{l}+(2l-k)e_{1})!}{|\beta^{l}+(2l-k)e_{1}|!}|w|^{2|\beta^{l}+(2l-k)e_{1}|}\Big\}\\
&:=&\frac{1}{(1-|w|^2)^n}+\sum\limits_{l=2}^{\infty} g_l(w).
\end{eqnarray*}
Since $|\beta^{l}|>n-2$ for all $l\geq1$, we have
$$\begin{array}{lll}
&&|g_{l}(w)|\\&=&\frac{(n+|\beta^{l}|-1)!}{|\beta^{l}|!(n-1)!}|w|^{2|\beta^{l}|}\left | \sum\limits_{k=2}^{l}\frac{(n+|\beta^{l}+ke_{1}|-1)!|\beta^{l}|!}{(n+|\beta^{l}|-1)!|\beta^{l}+ke_{1}|!}\left (\frac{1}{k}-1 \right)|w|^{2k}+\sum\limits_{k=2}^{l-1}\frac{(n+|\beta^{l}+(2l-k)e_{1}|-1)!|\beta^{l}|!}{(n+|\beta^{l}|-1)!|\beta^{l}+(2l-k)e_{1}|!}\left (\frac{1}{k}-1 \right)|w|^{4l-2k}\right |\\
&\leq&\frac{(n+|\beta^{l}|-1)!}{|\beta^{l}|!(n-1)!}|w|^{2|\beta^{l}|}\left |\sum\limits_{k=2}^{l}2^{k}\left(\frac{1}{k}-1\right)|w|^{2k}+\sum\limits_{k=2}^{l-1}2^{2l-k}\left (\frac{1}{k}-1\right)|w|^{4l-2k}\right |
\end{array}$$
and $M_l=\sup\limits_{|w|<1} \left |\sum\limits_{k=2}^{l}2^{k}\left(\frac{1}{k}-1\right)|w|^{2k}+\sum\limits_{k=2}^{l-1}2^{2l-k}\left (\frac{1}{k}-1\right)|w|^{4l-2k}\right |<2^{2l-1}.$
Setting $f(x)=x^{|\beta^{l}|}(1-x)^{n}$ for $0\leq x\leq 1$,
then $f(x)$ attains a maximum of $\left (\frac{|\beta^{l}|}{|\beta^{l}|+n} \right)^{|\beta^{l}|}\left (\frac{n}{|\beta^{l}|+n} \right)^{n}$ at $x=\frac{|\beta^{l}|}{|\beta^{l}|+n}.$
Therefore, from $|\beta^{l}|>\frac{n^{n}2^{3l+1}}{(n-1)!}-n,$ for all $w\in\mathbb{B}^{m},$
\begin{tiny}
\begin{equation*}
|g_{l}(w)|(1-|w|^{2})^{n}\leq M_{l}\frac{(n+|\beta^{l}|-1)!}{|\beta^{l}|!(n-1)!}|w|^{2|\beta^{l}|}(1-|w|^{2})^{n}
\leq2^{2l-1}\frac{(n+|\beta^{l}|-1)!}{|\beta^{l}|!(n-1)!}\frac{|\beta^{l}|^{|\beta^{l}|}n^{n}}{(|\beta^{l}|+n)^{|\beta^{l}|+n}}
\leq\frac{2^{2l-1}n^{n}}{(n-1)!(|\beta^{l}|+n)}
\leq\frac{1}{2^{l+2}}.
\end{equation*}
\end{tiny}
Since $K(w,w)(1-|w|^{2})^{n}=1+\sum\limits_{l=2}^{\infty}g_{l}(w)(1-|w|^{2})^{n},$ we know that
$\frac{7}{8} < K(w,w)(1-|w|^{2})^{n} <\frac{9}{8}.$
So $K(w,w)(1-|w|^2)^n$ is bounded and positive.
For any $l>2$ and $\alpha=\beta^{l}+(l-1)e_{1}\in\mathbf{Z}_{+}^{m}$, suppose $\alpha_{1}\neq0$ and $\alpha_{i}=0$ for $2\leq i\leq m$, then
$$\frac{\widetilde{\rho}(\alpha-e_{1})}{\widetilde{\rho}(\alpha)}=
\frac{\widetilde{\rho}(\beta^{l}+(l-2)e_{1})}{\widetilde{\rho}(\beta^{l}+(l-1)e_{1})}=\frac{l-1}{l-2}\frac{|\alpha|}{|\alpha|+n-1}>\frac{|\alpha|}{|\alpha|+n-1}$$
and
$\frac{\prod\limits_{k=1}^{l-1}\frac{\widetilde{\rho}(\beta^{l}+ke_{1})}{\widetilde{\rho}(\beta^{l}+(k+1)e_{1})}}{\prod\limits_{k=1}^{l-1}\frac{\rho(\beta^{l}+ke_{1})}{\rho(\beta^{l}+(k+1)e_{1})}}=
\frac{\rho(\beta^{l}+le_{1})}{\widetilde{\rho}(\beta^{l}+le_{1})}=l\rightarrow+\infty,$
as $l\rightarrow+\infty$. Thus, from Theorem \ref{c4.2} and Theorem \ref{thm2} , $\mathbf{T}$ is not an $n$-hypercontraction and $\mathbf{T}$ is not similar to $\mathbf{S}^*$.
\end{proof}

\end{document}